\documentclass[11pt,longtable,letterpaper,oneside,reqno]{amsart}
\usepackage{amsmath,amsfonts,amsthm,amssymb,amscd,mathtools,stmaryrd,url}
\usepackage{bbm}
\usepackage{epsfig}
\usepackage{graphicx}
\usepackage{latexsym}
\usepackage{longtable}
\usepackage{float}
\usepackage{hyperref}
\usepackage{color}

\DeclareFontFamily{OT1}{rsfs}{}
\DeclareFontShape{OT1}{rsfs}{n}{it}{<-> rsfs10}{}
\DeclareMathAlphabet{\mathscr}{OT1}{rsfs}{n}{it}

\newtheorem{prop}{Proposition}[section]
\newtheorem{thm}[prop]{Theorem}

\newtheorem{cor}[prop]{Corollary}
\newtheorem{lem}[prop]{Lemma}

\newtheorem {defn }{Definition}

\numberwithin{equation}{section}
\usepackage[margin=1in]{geometry}
\usepackage{lipsum}

\begin{document}
\title{Sums of Fibonacci numbers close to a power of 2}
\author[E. Hasanalizade]{Elchin Hasanalizade}
\address{Department of Mathematics and Computer Science, University of Lethbridge, 4401 University Drive, Lethbridge, Alberta, T1K 3M4 Canada}
\email{e.hasanalizade@uleth.ca}
\begin{abstract}
In this paper, we find all sums of two Fibonacci numbers which are close to a power of 2. As a corollary, we also determine all Lucas numbers close to a power of 2. The main tools used in this work are lower bounds for linear forms in logarithms due to Matveev and Dujella-Peth\"{o} version of the Baker-Davenport reduction method in diophantine approximation. This paper continues and extends the previous work of Chern and Cui.
\end{abstract}

\subjclass{11B39, 11J86}
\keywords{\noindent Fibonacci number, linear forms in logarithms, reduction method}
\date{\today}
\maketitle

\section{Introduction}

The Fibonacci sequence $(F_n)_{n\ge0}$ is the binary recurrence sequence defined by $F_0=F_1=1$ and 
\begin{align*}
F_{n+2}=F_{n+1}+F_n \ \text{for all} \ n\ge0.
\end{align*}

There is a broad literature on the Diophantine equations involving the Fibonacci numbers. Bugeaud, Mignotte and Siksek \cite{BMS} proved that the only perfect powers among the Fibonacci numbers are $F_0=0$, $F_1=F_2=1$, $F_6=8$ and $F_{12}=144$. In particular, 1, 2 and 8 are the only powers of 2 in the Fibonacci sequence. Bravo and Luca \cite{BL} studied the Diophantine equation 
\begin{align*}
F_n+F_m=2^a
\end{align*}
in nonnegative integers $n, m$ and $a$ with $n\ge m$. 
Bravo and Bravo \cite{BB} determined all the solutions of the Diophantine equation 
\begin{align*}
F_n+F_m+F_l=2^a
\end{align*}
in nonnegative integers $n, m, l$ and $a$ with $n\ge m\ge l$.

Chim and Ziegler \cite{CZ} completely solved the Diophantine equations 
\begin{align*}
F_{n_1}+F_{n_2}=2^{a_1}+2^{a_2}+2^{a_3}
\end{align*}
and 
\begin{align*}
F_{m_1}+F_{m_2}+F_{m_3}=2^{t_1}+2^{t_2}
\end{align*}
in nonnegative integers $n_1, n_2, m_1, m_2, m_3, a_1, a_2, a_3, t_1, t_2$  with $n_1\ge n_2$, $a_1\ge a_2\ge a_3$, $m_1\ge m_2\ge m_3$ and $t_1\ge t_2$,

We say that an integer $n$ is close to a positive integer $m$, if it satisfies 
\begin{align*}
|n-m|<\sqrt{m}.
\end{align*}
In 2014, Chern and Cui \cite{CC} studied the Fibonacci numbers which are close to a power of 2. More precisely, they found that the only solutions $(F_n, 2^m)$ of the Diophantine inequality 
\begin{align}
\label{Eq1}
|F_n-2^m|<2^{m/2}
\end{align}
are $(1, 2)$, $(2, 2)$, $(3, 2)$, $(3, 4)$, $(5, 4)$, $(8, 8)$, $(13, 16)$ and $(34, 32)$. 

In 2020, Bravo et al. \cite{BGH} extended the previous work \cite{CC} and found all the members of the $k$-generalized Fibonacci sequence $(F^{(k)}_n)_{n\ge-(k-2)}$, $k\ge2$ which are close to a power of 2.

In the present paper, we study the Diophantine inequality
\begin{align}
\label{Eq2}
|F_n+F_m-2^a|<2^{a/2}
\end{align}
in positive integers $n, m$ and $a$ with $n\ge m$. 

In particular, we prove the following main theorem. 
\begin{thm}
There are exactly 52 solutions $(n,m,a)\in\mathbb{N}^3$ to Diophantine inequality \eqref{Eq2}. All solutions satisfy $n\le42$ and $a\le28$. A list of solutions is given in the appendix. 
\end{thm}

The well-known companion sequence of the Fibonacci sequence is Lucas sequence $(L_n)_{n\ge0}$. It satisfies the same recursive relation as the Fibonacci numbers, but with initial conditions $L_0=2$ and $L_1=1$. Recall the following relation between the Fibonacci numbers and the Lucas numbers
\begin{align*}
L_k=F_{k-1}+F_{k+1} \ \text{for all} \ k\ge1.
\end{align*}
As a corollary we determine all Lucas numbers which are close to a power of 2.
\begin{cor} 
There are only 9 Lucas numbers which are close to a power of 2. Namely, the solutions $(n,a)\in\mathbb{N}^2$ of the inequality
\begin{align*}
|L_n-2^a|<2^{a/2}
\end{align*}
are $(1,1)$, $(2, 1)$, $(2, 2)$, $(3, 2)$, $(4, 3)$, $(6, 4)$, $(7, 5)$, $(10, 7)$ and $(13, 9)$. 
\end{cor}

We shall prove Theorem 1.1 by following closely the typical strategy as performed in \cite{BB}, \cite{BL}, \cite{CZ}. First, we apply a lower bound for linear forms in logarithms due to Matveev to obtain a large upper bound for $n$. Applying a version of the Baker-Davenport reduction method we reduce this huge bound.

\section{Preliminaries}

First, recall that the Binet formula for Fibonacci numbers 

\begin{align}
\label{Eq3}
F_n=\frac{\alpha^n-\beta^n}{\alpha-\beta}
\end{align}
holds for all $n\ge0$, where $\alpha=\frac{1+\sqrt{5}}{2}$ and $\alpha=\frac{1-\sqrt{5}}{2}=-\frac{1}{\alpha}$ are the roots of the characteristic equation $x^2-x-1=0$ of the Fibonacci sequence. Moreover, the equality
\begin{align}
\label{Eq4}
\alpha^{n-2}\le F_n\le \alpha^{n-1}
\end{align}
holds for all $n\ge1$.

Next we recall a definition and some facts about logarithmic height. Let $\alpha$ be an algebraic number of degree $d\ge1$ with the minimal polynomial
\begin{align*} 
a_dX^d+\ldots+a_1X+a_0=a_d\prod_i^d(X-\alpha_i),
\end{align*}
where $a_0,a_1,\ldots,a_d$ are relatively prime integers and $\alpha_1,\ldots,\alpha_d$ are the conjugates of $\alpha$. The absolute logarithmic Weil height of $\alpha$ is defined as
\begin{align*}
h(\alpha)=\frac{1}{\alpha}\big(\log{|a_d|}+\sum_i^d\log{(\text{max}\{|\alpha_i|,1\})}\big)
\end{align*}
In particular, if $\alpha=\frac{p}{q}$ is a rational number with $\text{gcd}(p,q)=1$ and $q>0$, then $h(\alpha)=\log{\text{max}\{|p|,q\}}$. The following properties of $h(\alpha)$ are well-known. 
\begin{itemize} 
    \item $h(\alpha\pm\beta)\le h(\alpha)+h(\beta)+\log{2}$,
    \item $h(\alpha\beta^{\pm1})\le h(\alpha)+h(\beta)$,
    \item $h(\alpha^l)=|l|h(\alpha)$, $l\in\mathbb{Z}$.
\end{itemize}

One of the main tools in this paper is a widely used estimate on lower bounds for linear forms in complex logarithms due to Matveev \cite{M}. 
\begin{thm}
\label{Matveev}
Let $\gamma_1,\ldots,\gamma_t$ be positive real algebraic numbers in a real algebraic number field $\mathbb{K}$ of degree $D$, let $b_1,\ldots,b_t\in\mathbb{Z}$ and 
\begin{align*}
\Lambda\coloneqq\gamma^{b_1}_1\cdots\gamma^{b_t}_t-1
\end{align*}
is non-zero. Then 
\begin{align*}
|\Lambda|>\text{exp}(-1.4\times30^{t+3}\times t^{4.5}\times D^2(1+\log{D})(1+\log{B})A_1\ldots A_t)
\end{align*}
where 
\begin{align*}
B\ge\text{max}\{|b_1|,\cdots,|b_t|\}
\end{align*}
and 
\begin{align*}
A_i\ge\text{max}\{Dh(\gamma_i),|\log{\gamma_i}|,0.16\} \ \text{for all} \ i=1,\ldots,t.
\end{align*}
\end{thm}

During the calculations, we get an upper bound on $n$ which is too large, thus we need to reduce it. To do so, we use some results from the theory of continued fractions and Diophantine approximation. The next lemma is known as the Legendre's criterion 
\begin{lem}
\label{Legendre}
Let $\tau$ be an irrational number, $\frac{p_0}{q_0},\frac{p_1}{q_1},\frac{p_2}{q_2},\ldots$ be all the convergents of the continued fraction expansion of $\tau$ and $M$ be a positive integer. Let $N$ be a nonnegative integer such that $q_N>M$. Then putting $a(M)\coloneqq\text{max}\{a_i:\ i=0,1,2,\cdots,N\}$, the inequality
\begin{align*}
\big|\tau-\frac{r}{s}\big|>\frac{1}{(a(M)+2)s^2}
\end{align*}
holds for all pairs $(r,s)$ of positive integers with $0<s<M$.
\end{lem}

At last, we need a variant of the Baker-Davenport lemma, which is due to Dujella and Peth\"{o} \cite{DP}. For a real number $X$, we denote by $||X||=\text{min}\{|x-n|,\ n\in\mathbb{Z}\}$ the distance from $X$ to the nearest integer.
\begin{lem}
\label{DP}
Let $M$ be a positive integer, $\frac{p}{q}$ be a convergent of the continued fraction expansion of the irrational number $\gamma$ such that $q>6M$ and $A,B,\mu$ be some real numbers with $A>0$ and $B>1$. Furthermore, let $\epsilon\coloneqq||\mu q||-M||\gamma q||$. If $\epsilon>0$, then there is no solution to the inequality 
\begin{align*}
0<|u\gamma-v+\mu|<AB^{-w}
\end{align*}
in positive integers $u,v$ and $w$ with 
\begin{align*}
u\le M \ \text{and} \ w\ge\frac{\log{\big(\frac{Aq}{\epsilon}\big)}}{\log{B}}.
\end{align*}
\end{lem}

\section{Proof}

\subsection{Upper bound for $n$}

First of all, note that if $(n,1,a)$ is a solution of \eqref{Eq2}, then $(n,2,a)$ is also a solution to \eqref{Eq2} since $F_1=F_2=1$. So from now on we assume that $m\ge2$. Due to the work of Chern and Cui all solutions of \eqref{Eq2} with $n=m$ are listed in the appendix. Since $F_n+F_{n-1}=F_{n+1}$, we can assume $n>m+1$ and in particular $n-m\ge2$. 
\begin{prop}
There are exactly 52 solutions $(n,m,a)\in\mathbb{N}^3$ to \eqref{Eq2} with $n\le250$. All solutions fulfill $n\le42$ and $a\le28$. The list of solutions is given in the appendix. 
\end{prop}
\begin{proof}
The solutions were found by a brute force search with a simple Python code.
\end{proof}

Because of Proposition 3.1 for the rest of the paper we assume that $n>250$ and we will show that there exist no solutions with $n>250$.

Observe that the Binet formula \eqref{Eq4} yields for $n>1$ the inequalities
\begin{align}
\label{Eq5}
0.38\alpha^n<\alpha^n\frac{1-\alpha^{-4}}{\sqrt{5}}\le F_n=\alpha^n\frac{1-(-1)^n\alpha^{-2n}}{\sqrt{5}}\le\alpha^n\frac{1-\alpha^{-6}}{\sqrt{5}}<0.48\alpha^n
\end{align}

Due to \eqref{Eq5} for $n>m>1$ we have 
\begin{align}
\label{Eq6}
0.38\alpha^n<F_n<F_n+F_m<0.48\alpha^n+0.48\alpha^{n-1}<0.78\alpha^n.
\end{align}

Let us now establish a relation between $n$ and $a$. Without loss of generality, we may assume that $a\ge2$. Combining \eqref{Eq2} with the right inequality of \eqref{Eq6}, one gets
\begin{align*}
2^{a-1}\le2^a-2^{\frac{a}{2}}<F_n+F_m<0.78\alpha^n<\alpha^n.
\end{align*}

On ther hand, combining \eqref{Eq2} with the left inequality of \eqref{Eq6}, we have 
\begin{align*}
0.38\alpha^n<F_n+F_m<2^a+2^{\frac{a}{2}}<2^{a+1}.
\end{align*}

Thus 
\begin{align}
\label{Eq7}
n\frac{\log{\alpha}}{\log{2}}+\frac{\log{0.38}}{\log{2}}-1<a<n\frac{\log{\alpha}}{\log{2}}+1,
\end{align}
where $\frac{\log{\alpha}}{\log{2}}=0.6942...$. In particular, we have $a<n$.

By the Binet formula, we have 
\begin{align}
\label{Eq8}
\frac{\alpha^n+\alpha^m}{\sqrt{5}}-(F_n+F_m)=\frac{\beta^n+\beta^m}{\sqrt{5}}
\end{align}

Taking the absolute values in the above equality and combining it with \eqref{Eq2}, we get 
\begin{align*}
\bigg|\frac{\alpha^n(1+\alpha^{m-n})}{\sqrt{5}}-2^a\bigg|<\frac{|\beta|^n+|\beta|^m}{\sqrt{5}}+2^{a/2}<\frac{1}{3}+2^{a/2}<2^{\frac{a}{2}+1}
\end{align*}
for all $n\ge4$ and $m\ge2$. Dividing both sides of the above inequality by $2^a$ we obtain 
\begin{align}
\label{Eq9}
\bigg|\frac{\alpha^n(1+\alpha^{m-n})}{2^a\sqrt{5}}-1\bigg|<2^{-\frac{a}{2}+1}.
\end{align}

In the first application of Matveev's theorem, we take the parameters $t=3$ and 
\begin{align*}
(\gamma_1,b_1)\coloneqq(2,-a), (\gamma_2,b_2)\coloneqq(\alpha,n), (\gamma_3,b_3)\coloneqq\bigg(\frac{1+\alpha^{m-n}}{\sqrt{5}},1\bigg)
\end{align*}

The three numbers $\gamma_1,\gamma_2,\gamma_3$ are real, positive and are contained in $K=\mathbb{Q}(\sqrt{5})$, so we can take $D\coloneqq[K:\mathbb{Q}]=2$.

Note that the left-hand side of \eqref{Eq9} is not zero, for otherwise we would get the relation
\begin{align}
\label{Eq10}
2^a\sqrt{5}=\alpha^n+\alpha^m.
\end{align}
Conjugating the above relation in $\mathbb{Q}(\sqrt{5})$, we get 
\begin{align}
\label{Eq11}
-2^a\sqrt{5}=\beta^n+\beta^m.
\end{align}
Combining \eqref{Eq10} and \eqref{Eq11}, we get 
\begin{align*}
\alpha^n<\alpha^n+\alpha^m=|\beta^n+\beta^m|\le|\beta|^n+|\beta|^m<1
\end{align*}
which is impossible for positive $n$. Hence the left-hand side of \eqref{Eq9} is non-zero.

Since $h(\gamma_1)=\log{2}=0.6931...$ we can choose $A_1\coloneqq1.4>Dh(\gamma_1)$. Further since $h(\gamma_2)=\frac{\log{\alpha}}{2}=0.2406...$ we can take $A_2\coloneqq0.5>Dh(\gamma_2)$. Let us now estimate $h(\gamma_3)$. We begin by observing that
\begin{align*}
\gamma_3=\frac{1+\alpha^{m-n}}{\sqrt{5}}<\frac{2}{\sqrt{5}} \ \text{and} \ \gamma^{-1}_3=\frac{\sqrt{5}}{1+\alpha^{m-n}}<\sqrt{5}
\end{align*}
so that $|\log{\gamma_3}|<1$. Next, by the properties of the logarithmic height we have 
\begin{align*}
h(\gamma_3)\le \log{\sqrt{5}}+|m-n|\bigg(\frac{\log{\alpha}}{2}\bigg)+\log{2}=\log{2\sqrt{5}}+(n-m)\bigg(\frac{\log{\alpha}}{2}\bigg)
\end{align*}
Hence, we can take
\begin{align*}
A_3\coloneqq3+(n-m)\log{\alpha}>\text{max}\{2h(\gamma_3),|\log{\gamma_3}|,0.16\}.
\end{align*}

Finally, by recalling that $a<n$, we infer that $\text{max}\{|b_1|,|b_2|,|b_3|\}=n$, so we can take $B\coloneqq n$. We put $\Lambda_1\coloneqq\frac{\alpha^n(1+\alpha^{m-n})}{2^a\sqrt{5}}-1$. Now Matveev's theorem tells us that the left-hand side of \eqref{Eq9} is bounded below by
\begin{align*}
\log{|\Lambda_1|}&>-1.4\times30^6\times3^{4.5}\times2^2(1+\log{2})(1+\log{n})\times1.4\times0.5\times(3+(n-m)\log{\alpha})\\ 
&>-1.4\times10^{12}\times\log{n}\times(3+(n-m)\log{\alpha})
\end{align*}
where we used the fact that the inequality $1+\log{n}<2\log{n}$ holds for all $n\ge3$.

By comparing the above inequality with the right-hand side of \eqref{Eq9} we get that
\begin{align}
\label{Eq12}
\bigg(\frac{a}{2}-1\bigg)\log{2}<1.4\times10^{12}\times\log{n}\times(3+(n-m)\log{\alpha}).
\end{align}

Let us consider a second linear form in logarithms. To this end, we rewrite \eqref{Eq8} as follows
\begin{align*}
\frac{\alpha^n}{\sqrt{5}}-(F_n+F_m)=\frac{\beta^n}{\sqrt{5}}-F_m
\end{align*}
Again combining the above relation with \eqref{Eq2}, we get 
\begin{align*}
\bigg|\frac{\alpha^n}{\sqrt{5}}-2^a\bigg|<2^{\frac{a}{2}}+\frac{|\beta|^n}{\sqrt{5}}+F_m<2^{\frac{a}{2}}+\frac{1}{2}+\alpha^m,
\end{align*}
where we have also used the fact that $|\beta|^n<1/2$ for all $n\ge2$. Dividing both sides of the above expression by $\frac{\alpha^n}{\sqrt{5}}$ and taking into account that $\alpha>\sqrt{2}$ and $n>m$, we obtain
\begin{align}
\label{Eq13}
|1-2^a\cdot\alpha^{-n}\cdot\sqrt{5}|<\frac{2^{\frac{a}{2}}\sqrt{5}}{\alpha^n}+\frac{\sqrt{5}}{2\alpha^n}+\frac{\sqrt{5}}{\alpha^{n-m}}<\frac{3\sqrt{5}}{2}\text{max}\{\alpha^{m-n},\alpha^{a-n}\}
\end{align}

In a second application of Matveev's theorem, we take the parameters $t=3$ and 
\begin{align*}
(\gamma_1,b_1)\coloneqq(2, a), (\gamma_2,b_2)\coloneqq(\alpha, -n), (\gamma_3,b_3)\coloneqq(\sqrt{5}, 1)
\end{align*}

The three numbers $\gamma_1,\gamma_2,\gamma_3$ are real, positive and belong to $K=\mathbb{Q}(\sqrt{5})$, so we can take $D=2$. Put $\Lambda_2\coloneqq1-2^a\cdot\alpha^{-n}\cdot\sqrt{5}$. If $\Lambda_2=0$, then we have that $2^a=\frac{\alpha^n}{\sqrt{5}}$, so $\alpha^{2n}\in\mathbb{Z}$, which is impossible. Thus $\Lambda_2\ne0$. Since $h(\gamma_1)=\log{2}=0.6931...$ we can choose $A_1\coloneqq1.4>Dh(\gamma_1)$. Further since $h(\gamma_2)=\frac{\log{\alpha}}{2}=0.2406...$ and $h(\gamma_3)=\log{\sqrt{5}}=0.8047...$, it follows that we can take $A_2\coloneqq0.5>Dh(\gamma_2)$ and $A_3\coloneqq1.7>Dh(\gamma_3)$. Finally, since $a<n$, we deduce that $B\coloneqq\text{max}\{|b_1|,|b_2|,|b_3|\}=n$.

Then Matveev's theorem tells us that 
\begin{align*}
\log{|\Lambda_2|}&>-1.4\times30^6\times3^{4.5}\times2^2(1+\log{2})(1+\log{n})\times1.4\times0.5\times1.7\\
&>-2.31\times10^{12}\log{n}
\end{align*}

Comparing the resulting inequality with \eqref{Eq13}, we get 
\begin{align*}
\text{min}\{(n-a)\log{\alpha},(n-m)\log{\alpha}\}<2.4\times10^{12}\log{n}.
\end{align*}

Now the argument splits into two cases.

{\bf Case 1.} $\text{min}\{(n-a)\log{\alpha},(n-m)\log{\alpha}\}=(n-m)\log{\alpha}$.

In this case we have 
\begin{align*}
\bigg(\frac{a}{2}-1\bigg)\log{2}&<1.4\times10^{12}\log{n}(3+(n-m)\log{\alpha})\\ 
&<1.4\times10^{12}\log{n}(3+2.4\times10^{12}\log{n})<3.5\times10^{24}\log^2{n}
\end{align*}
Combining it with the left inequality of \eqref{Eq7} and by a calculation in {\it Mathematica}, we obtain 
\begin{align*}
a<4.5\times10^{28} \ \text{and} \ n<6.6\times10^{28}.
\end{align*}

{\bf Case 2.} $\text{min}\{(n-a)\log{\alpha},(n-m)\log{\alpha}\}=(n-a)\log{\alpha}$.

In this case we have
\begin{align*}
(n-a)\log{\alpha}<2.4\times10^{12}\log{n}.
\end{align*}

Note that the right inequality of \eqref{Eq7} yields 
\begin{align*}
n-a>n\frac{(1-\log{\alpha})}{\log{2}}-1.
\end{align*}

Combining the above inequalities and by a calculation in {\it Mathematica}, we get 
\begin{align*}
a<2.3\times10^{14} \ \text{and} \ n<1.6\times10^{14}.
\end{align*}

Thus, in both Case 1 and Case 2, we have 
\begin{align}
\label{Eq14}
a<4.5\times10^{28} \ \text{and} \ n<6.6\times10^{28}.
\end{align}

We now need to reduce the above bound for $n$.

\subsection{Reduction of the bound}

We begin by recalling the following facts. For any non-zero real number $x$, we have

i) $0<x<e^x-1$,

ii) if $x<0$ and $|e^x-1|<1/2$, then $|x|<2|e^x-1|$.

We may assume that $n-m>250$ and $n-a>250$. We go back to the inequality \eqref{Eq13}. Let 
\begin{align*}
z_1\coloneqq a\log{2}-n\log{\alpha}+\log{\sqrt{5}}.
\end{align*}

Since we assume that $\text{min}\{n-m,n-a\}>250$ we get $|e^{z_1}-1|<1/2$ and by ii) we have 
\begin{align*}
|a\log{2}-n\log{\alpha}+\log{\sqrt{5}}|<\frac{3\sqrt{5}}{\alpha^{\text{min}\{n-m,n-a\}}}.
\end{align*}

Dividing both sides of the above inequality by $\log{\alpha}$, we conclude that 
\begin{align*}
0<\bigg|a\frac{\log{2}}{\log{\alpha}}-n+\frac{\sqrt{5}}{\log{\alpha}}\bigg|<\frac{3\sqrt{5}}{\log{\alpha}}\cdot\alpha^{-\kappa},
\end{align*}
where $\kappa=\text{min}\{n-m,n-a\}$. We apply Lemma \ref{DP} with 
\begin{align*}
\gamma\coloneqq\frac{\log{2}}{\log{\alpha}}, \ \mu\coloneqq\frac{\sqrt{5}}{\log{\alpha}}, \ A\coloneqq\frac{3\sqrt{5}}{\log{\alpha}}, \ B\coloneqq\alpha.
\end{align*}

Clearly $\gamma$ is an irrational number. We also take $M=4.5\times10^{28}$ which is an upper bound for $a$. We find that the convergent $\frac{p}{q}=\frac{p_{67}}{q_{67}}$ is such that $q>6M$. By using this we have that $\epsilon>0.01038$, therefore either
\begin{align*}
n-m<\frac{\log{\bigg(\frac{3\sqrt{5}q}{0.01038\log{\alpha}}\bigg)}}{\log{\alpha}}<158 \ \text{or} \ n-a<\frac{\log{\bigg(\frac{3\sqrt{5}q}{0.01038\log{\alpha}}\bigg)}}{\log{\alpha}}<158
\end{align*}

Thus, we have that either $n-m<158$ or $n<213$. The latter case contradicts our assumption that $n>250$. Inserting the upper bound for $n-m$ into \eqref{Eq12} we get that $a<1.3\times10^{16}$.

Let us now work on the inequality \eqref{Eq9}. We may assume that $a>28$. By i) and ii) it becomes 
\begin{align*}
|a\log{2}-n\log{\alpha}+\log{\phi(n-m)}|<\frac{4}{2^{a/2}}
\end{align*}
where $\phi$ is defined by $\phi(t)\coloneqq\sqrt{5}(1+\alpha^{-t})^{-1}$. Again dividing both sides by $\log{\alpha}$, we obtain that 
\begin{align}
\label{Eq15}
0<\bigg|a\frac{\log{2}}{\log{\alpha}}-n+\frac{\log{\phi(n-m)}}{\log{\alpha}}\bigg|<\frac{4}{\log{\alpha}}\cdot2^{-a/2}
\end{align}

Here we take $M=1.3\times10^{16}$ (an upper bound for $a$) and apply Lemma \ref{DP} with 
\begin{align*}
\gamma\coloneqq\frac{\log{2}}{\log{\alpha}}, \ \mu\coloneqq\frac{\phi(n-m)}{\log{\alpha}}, \ A\coloneqq\frac{4}{\log{\alpha}}, \ B\coloneqq\sqrt{2}
\end{align*}
for all choices $n-m\in\{1,\ldots,158\}$ except when $n-m=2,6$. With the help of {\it Mathematica} we find that if $(n,m,a)$ is a possible solution of \eqref{Eq2} with $n-m\ne2,6$ then $a\le67$ and thus, $n\le100$. But this is a contradiction to our assumption that $n>100$.

Let us now consider the special cases when $n-m=2$ and $6$. Note that we cannot study these cases as before because the parameter $\mu$ appearing in Lemma \ref{DP} is 
\begin{align*} 
\frac{\log{\phi(t)}}{\log{\alpha}}=
 \left.
  \begin{cases}
   1 & \text{if } t=2 \\
   3-\frac{\log{2}}{\log{\alpha}} & \text{if } t=6
     \end{cases}
  \right.
\end{align*}
and the corresponding value of $\epsilon$ is always negative. When $n-m=2$ from \eqref{Eq15} we get that 
\begin{align}
\label{Eq16}
0<|a\gamma-(n-1)|<\frac{4}{\log{\alpha}}\cdot2^{-a/2}<\frac{4}{\log{\alpha}}\cdot2^{-\frac{1}{2}(n\frac{\log{\alpha}}{\log{2}}+\frac{\log{0.38}}{\log{2}}-1)}
\end{align}

Recall that $a<1.3\times10^{16}$. Let $[a_0,a_1,a_2,a_3,a_4,\ldots]=[1,2,3,1,2,\ldots]$ be the continued fraction of $\gamma$. A quick search using {\it Mathematica} reveals that 
\begin{align*}
q_{35}<1.3\times10^{16}<q_{36}.
\end{align*}

Furthermore, $a_M\coloneqq\text{max}\{a_i:\ i=0,1,\ldots,36\}=a_{17}=134$. So by Lemma \ref{Legendre} we have 
\begin{align}
\label{Eq17}
|a\gamma-(n-1)|>\frac{1}{(a_M+2)a}.
\end{align}

Comparing estimates \eqref{Eq16} and \eqref{Eq17} we get that $n<187$. In a similar manner one can get $n<187$ in the case when $n-m=6$. This again contadicts our assumption that $n>250$.

\section{Appendix}

The solutions for Diophantine inequality \eqref{Eq2} are displayed below 

\begin{center}
\scalebox{0.9}{
\begin{tabular}{ |c|c|c| } 
 \hline
 $|F_1+F_1-2|<\sqrt{2}$ & $|F_6+F_6-2^4|<2^{2}$ & $|F_{10}+F_7-2^6|<2^3$ \\ 
 \hline
 $|F_2+F_i-2|<\sqrt{2}$, $i=1,2$ & $|F_7+F_i-2^4|<2^{2}$, $i=1,2$ & $|F_{11}+F_9-2^7|<2^{7/2}$ \\ 
 \hline
 $|F_3+F_i-2|<\sqrt{2}$, $i=1,2$ & $|F_7+F_3-2^4|<2^{2}$ & $|F_{13}+F_6-2^8|<2^4$ \\ 
 \hline
 $|F_3+F_i-2^2|<2$, $i=1,2$ & $|F_7+F_4-2^4|<2^{2}$ & $|F_{13}+F_7-2^8|<2^4$\\ 
 \hline
 $|F_3+F_3-2^2|<2$ & $|F_7+F_5-2^4|<2^{2}$ & $|F_{13}+F_8-2^8|<2^4$\\
 \hline
 $|F_4+F_i-2^2|<2$, $i=1,2$ & $|F_8+F_6-2^5|<2^{5/2}$ & $|F_{13}+F_9-2^8|<2^4$\\
 \hline
 $|F_4+F_3-2^2|<2$ & $|F_8+F_7-2^5|<2^{5/2}$ & $|F_{14}+F_{12}-2^9|<2^{9/2}$\\
 \hline
 $|F_4+F_4-2^3|<2^{3/2}$ & $|F_9+F_i-2^5|<2^{5/2}$, $i=1,2$ & $|F_{16}+F_6-2^{10}|<2^5$\\
 \hline
 $|F_5+F_i-2^3|<2^{3/2}$, $i=1,2$ & $|F_9+F_3-2^5|<2^{5/2}$ & $|F_{16}+F_7-2^{10}|<2^5$\\
 \hline
 $|F_5+F_3-2^3|<2^{3/2}$ & $|F_9+F_4-2^5|<2^{5/2}$ & $|F_{16}+F_8-2^{10}|<2^5$\\
 \hline
 $|F_5+F_4-2^3|<2^{3/2}$ & $|F_9+F_9-2^6|<2^3$ & $|F_{16}+F_9-2^{10}|<2^5$\\
 \hline
 $|F_5+F_5-2^3|<2^{3/2}$ & $|F_{10}+F_3-2^6|<2^3$ & $|F_{16}+F_{10}-2^{10}|<2^5$\\
 \hline
  $|F_6+F_i-2^3|<2^{3/2}$, $i=1,2$ & $|F_{10}+F_4-2^6|<2^3$ & $|F_{23}+F_{19}-2^{15}|<2^{15/2}$\\
 \hline
 $|F_6+F_3-2^3|<2^{3/2}$ & $|F_{10}+F_5-2^6|<2^3$ & $|F_{42}+F_{29}-2^{28}|<2^{14}$\\
 \hline
 $|F_6+F_5-2^4|<2^{2}$ & $|F_{10}+F_6-2^6|<2^3$ & \\
 \hline
 \end{tabular}}
 \end{center}
\

\noindent {\bf Acknowledgement.} The author is grateful to Professor Florian Luca for his constructive comments which helped improve the presentation of this paper.


\end{document}